\documentclass[11pt,a4paper]{amsart}
\usepackage{amsmath,amssymb,amsfonts}
\usepackage[all]{xy}
\usepackage{caption}
\usepackage{enumerate}
\usepackage{hyperref} 
\usepackage{bm}
\usepackage{graphicx}
\usepackage{xcolor}
\usepackage{multicol}
\usepackage{tabularx}

\usepackage{soul}

\usepackage{mathpazo}
\usepackage{hyperref}
\usepackage{appendix}
\usepackage{a4wide}

\theoremstyle{plain}
\newtheorem{thm}{Theorem}[section]

\newtheorem{lem}[thm]{Lemma}
\newtheorem{cor}[thm]{Corollary}
\newtheorem{prop}[thm]{Proposition}

\theoremstyle{definition}
\newtheorem{definition}[thm]{Definition}
\newtheorem{ex}[thm]{Example}

\newtheorem{rem}[thm]{Remark}

\makeatletter
\newcommand{\neutralize}[1]{\expandafter\let\csname c@#1\endcsname\count@}
\makeatother

\newcommand{\C}{{\mathbb C}}

\newcommand{\F}{{\mathbb F}}

\newcommand{\Q}{{\mathbb Q}}
\newcommand{\R}{{\mathbb R}}

\newcommand{\Z}{{\mathbb Z}}

\begin{document}
\title{On the Paley graph of a quadratic character}
\dedicatory{Dedicated to Professor Moshe Rosenfeld  on the occasion of his 85th birthday}

\date{}

 \author{ J\'an Min\'a\v{c}, Lyle Muller, Tung T. Nguyen, Nguy$\tilde{\text{\^{E}}}$n Duy T\^{a}n }

\date{\today}

\address{Department of Mathematics, Western University, London, Ontario, Canada N6A 5B7}
\email{minac@uwo.ca}

\address{Department of Mathematics, Western University, London, Ontario, Canada N6A 5B7}
\email{lmuller2@uwo.ca}

\address{Department of Mathematics, Western University, London, Ontario, Canada N6A 5B7}
 \email{tungnt@uchicago.edu}
 
  \address{
 School of Applied  Mathematics and 	Informatics, Hanoi University of Science and Technology, 1 Dai Co Viet Road, Hanoi, Vietnam } 
\email{tan.nguyenduy@hust.edu.vn}


\thanks{Jan Minac is partially supported by the Natural Sciences and Engineering Research Council of Canada (NSERC) grant R0370A01. Jan Minac also gratefully acknowledges 
Faculty of Sciences Distinguished Research Professorship award for 2020/21. Jan Minac, Lyle Muller and Tung T Nguyen acknowledge the support of the Western Academy for Advanced Research. Nguy$\tilde{\text{\^{e}}}$n Duy T\^{a}n is is partially supported  by Vietnam National University, Hanoi (VNU) under the project number QG.23.48}

\maketitle 

\begin{abstract}

Paley graphs form a nice link between the distribution of quadratic residues and graph theory. These graphs possess remarkable properties which make them useful in several branches of mathematics. Classically, for each prime number $p$ we can construct the corresponding Paley graph using quadratic and non-quadratic residues modulo $p$. Therefore, Paley graphs are naturally associated with the Legendre symbol at $p$ which is a quadratic Dirichlet character of conductor $p$. In this article, we introduce the generalized Paley graphs. These are graphs that are associated with a general quadratic Dirichlet character. We will then provide some of their basic properties. In particular, we describe their spectrum explicitly. We then use those generalized Paley graphs to construct some new families of Ramanujan graphs. Finally, using special values of $L$-functions, we provide an effective upper bound for their Cheeger number. As a by-product of our approach, we settle a question raised in \cite{cramer2016isoperimetric} about the size of this upper bound.

\end{abstract}
\section{Introduction} 
Let $q=p^n$ be a prime power. The Paley graph $P_q$ is the graph with the following data 
\begin{enumerate}
    \item The vertex set of $P_q$ is $\F_q$; here $\F_q$ is the finite field with $q$ elements;
    \item Two vertices $u,v$ are connected if and only $u-v$ a nonzero square in $\F_q$; i.e., $v-u$ is an element of the set 
    \[ S = \{x^2\mid x \in \F_q, x \neq 0 \} .\]
\end{enumerate}
Let us recall the definition of the Caley graph $\Gamma(G,S)$ where $G$ is a group and a set $S$ is a subset of $G$ which has the property that $1 \not \in S$. The vertices of $\Gamma(G,S)$ are elements of $G$ and the edges of $\Gamma(G,S)$ are given by 
\[ E = \{ (g, gs)\mid g \in G, s \in S \} \subset G \times G.\]
(See \cite{cayley1878desiderata, krebs2011expander}). We note that in \cite[Definition 2.3.1]{kowalski2019introduction}, the author requires that $S$ is symmetric in the sense that $s \in S$ if and only if $s^{-1} \in S$ (this symmetric condition guarantees that $\Gamma(G,S)$ is an undirected graph). In the case $G=\F_q$ as an additive group, we see that $P_q$ is exactly the Caley graph $\Gamma(G,S)$ with $S=(\F_q^{\times})^2:=\{x^2\mid x\in \F_q^\times\}$. (Here we use the standard notation $\F_q^\times=\F_q\setminus\{0\}$.) We remark that $-1 \in S$ if and only if $q \equiv 1 \pmod{4}.$ Therefore, $P_{q}$ is an undirected graph for $q \equiv 1 \pmod{4}.$

The origin of Paley graphs can be traced back to Paley's 1933 paper \cite{paley1933orthogonal} in which he implicitly used them to construct Hadamard matrices. These graphs were later rediscovered by Carlitz in \cite{carlitz1960theorem}, though this time in a different context. Gareth A. Jones had the following remark in \cite{jones2020paley} about Paley graphs 
\begin{quotation}
Anyone who seriously studies algebraic graph theory or finite permutation groups will, sooner or later, come across the Paley graphs and their automorphism groups.
\end{quotation}
Due to its arithmetic and representation theoretic nature, we have various tools to study Paley graphs. Consequently, they have interesting properties that make them useful for graph theory and related fields. For example, Paley graphs have found applications in coding and cryptography theory (see \cite{ghinelli2011codes, javelle2014cryptographie}).

In this article, we define and study generalized Paley graphs which are associated with general quadratic characters (we refer the reader to Section \ref{section:paley} for the precise definition of these graphs). The terminology "generalized Paley graphs" was used earlier in some papers of C.E. Praeger and her collaborators. Their generalization of Paley graphs, however, went in another direction from our generalization (see \cite{LimPraeger2009GeneralizedPaley}). Therefore a possible confusion is unlikely. Using the theory of Gauss sums and circulant matrices, we describe explicitly the spectrum of these generalized Paley graphs. We then give a complete answer to the following question: which generalized Paley graphs are Ramanujan? In the final section, we provide an effective upper bound for the Cheeger number of these generalized Paley graphs using special values of $L$-functions. As a by-product of our approach, we settle a question raised in \cite[Section 2.2]{cramer2016isoperimetric} about the size of this upper bound. 

In closing the introduction, we remark that our interest in this topic arises naturally from our recent works on generalized Fekete polynomials (see \cite{MTT3, MTT4}), on the join of circulant matrices (see \cite{CM2, CM1}), and on some new constructions of Ramanujan graphs (see \cite{CM1_b}).

\section{Quadratic characters and the associated generalized Paley graphs} \label{section:paley}
\subsection{Kronecker symbol} 
In this section, we recall the definition of the Kronecker symbol. 

Let $a, n$ be integers. Suppose that $n\not=0$ and $n$ has the following factorization into the product of distinct prime numbers 
\[ n = \text{sgn}(n) p_1^{e_1} p_2^{e_2} \cdots p_r^{e_r} .\] 
Here $\text{sgn}(n)$ is the sign of $n$, which is $1$ if $n>0$ and $-1$ otherwise. Then, the Kronecker symbol $\left(\frac{a}{n} \right)$ is defined as 
\[ \left(\frac{a}{n} \right) = \left(\frac{a}{\text{sgn}(n)} \right) \left(\frac{a}{p_1} \right)^{e_1} \left(\frac{a}{p_2} \right)^{e_2} \cdots \left(\frac{a}{p_r} \right)^{e_r}.\]
Here \[ \left(\frac{a}{-1} \right) = \begin{cases}
  1  &  \text{if } a \geq 0 \\
  -1 &  \text{if } a<0 
\end{cases},  
\quad \left(\frac{a}{2} \right) = \begin{cases}
  0  &  \text{if $2|a$ }  \\
  1 &  \text{if } a \equiv \pm{1} \pmod{8} \\ 
  -1 & \text{if } a \equiv \pm{3} \pmod{8}, 
\end{cases} \] 
and for an odd prime $p$, $\left(\frac{a}{p} \right)$ is the Legendre symbol which is defined as follows 
\[ \left(\frac{a}{p} \right)  = \begin{cases}
   0 & \text{if $p|a$ } \\ 
  1  &  \text{if $a$ is a square modulo p } \\
  -1 &  \text{else. }
\end{cases} \]  
Also, the Kronecker symbol $\left(\frac{a}{0} \right)$ is defined as
\[ \left(\frac{a}{0} \right) = \begin{cases}
  1  &  \text{if } a = \pm{1} \\
  0 &  \text{otherwise. }
\end{cases} \]

We note that, the Kronecker symbol is defined for all integers $a$ and $n$. 
\subsection{Quadratic characters}
We first recall the theory of Dirichlet characters (we refer the reader to \cite[Chapter 1]{ayoub1963introduction} for some further discussions). 
\begin{definition}
    
    Let $m$ be a positive integer. A Dirichlet character mod $m$ is a group homomorphism 
    \[ \chi: (\Z/m)^{\times} \to \C^{\times}.\]

    We say that $\chi$ is primitive if it does not factor through $(\Z/n)^{\times}$ for some proper divisor $n$ of $m.$
\end{definition}
We remark that we can extend a Dirichlet character to an arithmetic function $\chi: \Z \to \C$ by the following convention 
\[ \chi(a) = \begin{cases}
  \chi( a \mod{m})  &  \text{if } \gcd(a,m)=1 \\
  0 &  \text {else.} 
\end{cases} \] 

In this article, we will focus on quadratic characters. These characters are closely related to the theory of quadratic fields and Kronecker symbols which we know recall. Let $m$ be a squarefree integer. Let $\Delta$ be the discriminant of the quadratic extension $\Q(\sqrt{m})/\Q$, which is given by 
\[ \Delta = \begin{cases}
  m  &  \text{if } m \equiv 1 \pmod{4} \\
  4m &  \text{if } m \equiv 2, 3 \pmod{4}.
\end{cases} \] 
From this definition, we can see that $\Delta$ necessarily determines the quadratic extension $\Q(\sqrt{m})/\Q.$ We have the following theorem.
\begin{thm} (\cite[Theorem 9.13]{montgomery2007multiplicative})
Let $\chi_{\Delta}\colon \Z \to \mathbb{C}^{\times}$ be the function given by 
\[ \chi_{\Delta}(a) = \left(\frac{\Delta}{a} \right) ,\] 
where $\left(\frac{\Delta}{a} \right)$ is the Kronecker symbol. Then $\chi_{\Delta}$ is a primitive quadratic character of conductor $D=|\Delta|.$ Furthermore, every primitive quadratic character is given uniquely this way. 
\end{thm}
\subsection{Generalized Paley graphs and their spectra}
Let $\chi = \chi_{\Delta}$ be the primitive quadratic character of conductor $D = |\Delta|$ as explained in the previous section. We introduce the following definition (see 
also \cite{budden2017dirichlet} for a similar approach).
\begin{definition} 
The generalized Paley graph $P_{\Delta}$ is the graph with the following data 
\begin{enumerate}
    \item The vertices of $P_{\Delta}$ are $\{0, 1, \ldots, D-1 \}.$
    \item Two vertices $(u,v)$ are connected iff $\chi_{\Delta}(v-u)=1.$
\end{enumerate}
\end{definition}
\begin{ex}
We refer to Figure \ref{fig:2figsA} and Figure \ref{fig:2figsB} for two concrete examples of generalized Paley graphs. 
\end{ex}
\begin{figure}[!htb]
\parbox{8cm}{
\includegraphics[scale = 0.3]{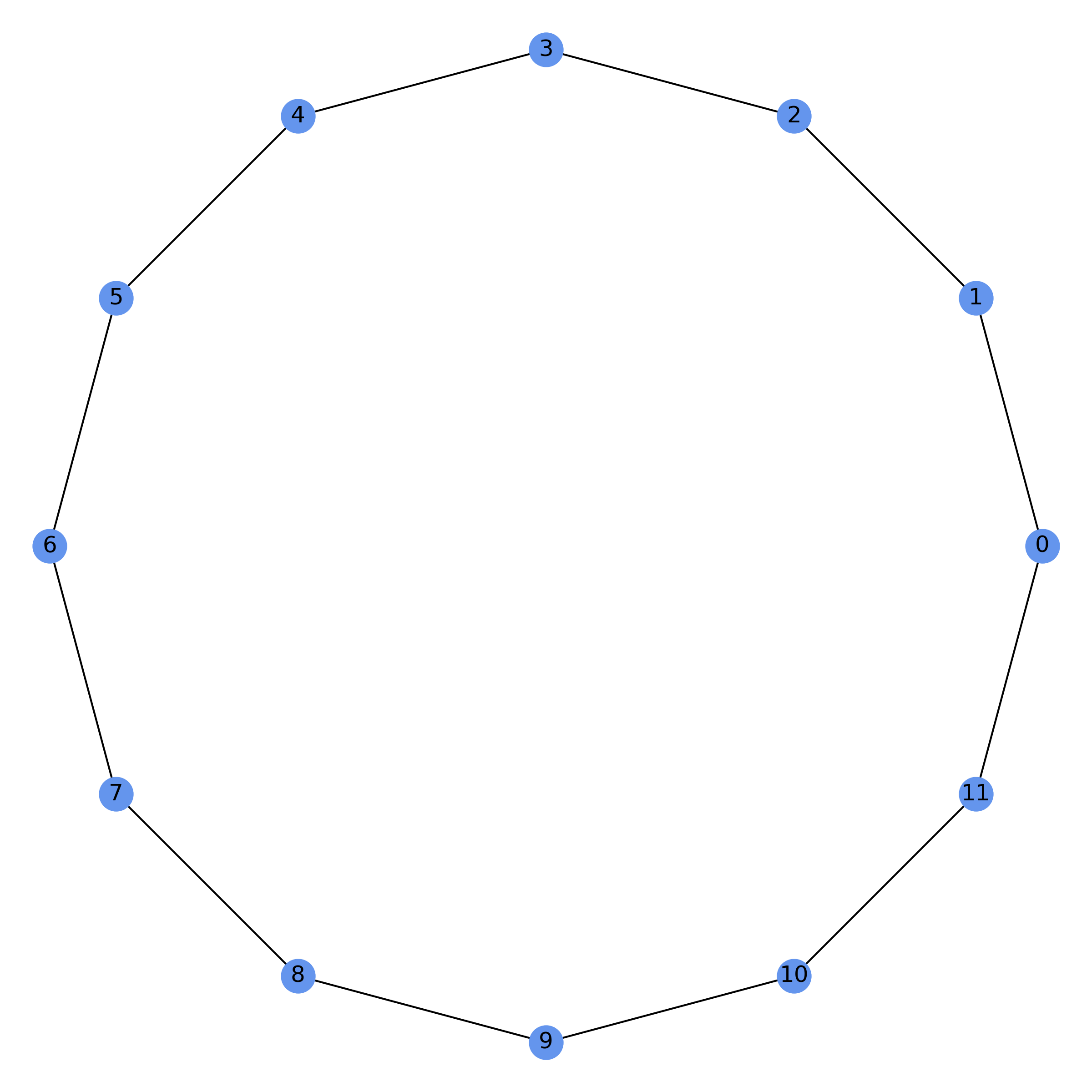}
\caption{The generalized Paley graph $P_{12}$}
\label{fig:2figsA}}
\qquad \quad
\begin{minipage}{6cm}
\includegraphics[scale = 0.3]{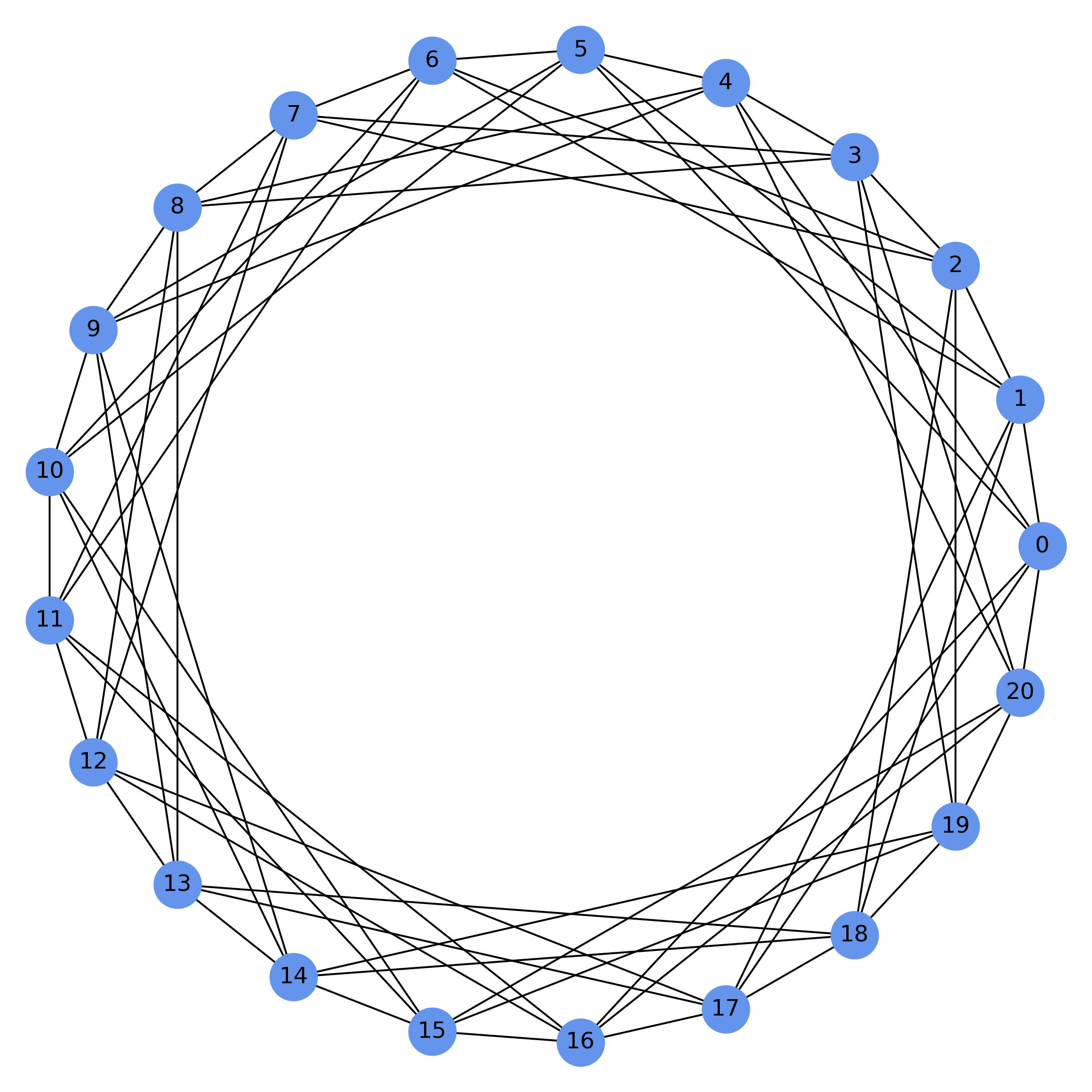}
\caption{The generalized Paley graph $P_{21}$.}
\label{fig:2figsB}
\end{minipage}
\end{figure}

We remark that if $\Delta<0$, then $\chi_{\Delta}(-1)=-1.$ In this case, $P_{\Delta}$ is a directed graph. Otherwise, if $\Delta>0$ then $\chi_{\Delta}(-1)=1$ and therefore, $P_{\Delta}$ is an undirected graph.

Since the connection in $P_{\Delta}$ is determined by $(v-u) \mod{D}$, we conclude that $P_{\Delta}$ is a circulant graph with respect to the cyclic group $\Z/D$ (see \cite{CM2, davis2013circulant}). In fact, its adjacency matrix is generated by the following vector
\[ v = \left[\frac{1}{2} \chi(a) (\chi(a)+1) \right]_{0 \leq a \leq D-1} .\] 
This follows from the fact that 
\[ \frac{1}{2} \chi(a) (1+\chi(a)) = \begin{cases}
  1  &  \text{ if } \chi(a) = 1 \\
  0 &  \text{ else.}
\end{cases} \]

We first observe the following.
\begin{prop} \label{prop:degree}
$P_{\Delta}$ is a regular graph of degree $\frac{1}{2}\varphi(D).$
\end{prop}
\begin{proof}
We have 
\begin{align*}
2 \deg (P_{\Delta}) &= \sum_{a=0}^{D-1}  \chi(a)[1+\chi(a)]  = \sum_{a=0}^{D-1} \chi(a) + \sum_{a=0}^{D-1} \chi^2(a) \\
&= 0 + \sum_{0 \leq a \leq D-1, \gcd(a,D)=1} 1 = \varphi(D).
\end{align*}
Here we use the fact that 
\[ \sum_{a=0}^{D-1} \chi(a)=0 .\] 
We conclude that $\deg(P_{\Delta})=\frac{1}{2} \varphi(D).$
\end{proof}

\begin{cor} \label{cor:cycle}
Suppose that $\Delta>0$. Then $P_{\Delta}$ is a cycle graph if and only if $\Delta=5$ or $\Delta=8$ or $\Delta=12.$
\end{cor}
\begin{proof}
Suppose that $P_{\Delta}$ is a cycle graph. Then the degree of $P_{\Delta}$ is $2$. By Proposition \ref{prop:degree}, we must have $\varphi(D)=4.$ Therefore $\Delta = D \in \{5,8, 12 \}.$ Conversely, if $\Delta \in \{5,8,12\}$, then $\chi(a)=1$ if and if $a = \pm{1}.$ Consequently, an edge in $P_{\Delta}$ must have the form $(u, u+1)$ or $(u, u-1)$ for $u \in \Z/D.$ In other words, $P_{\Delta}$ is a cycle graph.
\end{proof}

\section{The spectrum of $P_{\Delta}$}
In this section, we compute the spectrum of $P_{\Delta}$. By the Circulant Diagonalization Theorem (see \cite{davis2013circulant}), the spectrum of $P_{\Delta}$ is given by 
\[ \left \{ \lambda(\omega):= \frac{1}{2} \sum_{a=0}^{D-1} \chi(a) (1+\chi(a)) \omega^a \right\} ,\]
where $\omega$ runs over the set of all $D$th roots of unity. Note that for each $\omega$, there exists a unique positive integer $d \vert D$ such that $\omega$ is a primitive $d$th  root of unity. We first recall the definition of the M\"{o}bius function
\begin{definition}
The M\"{o}bius function $\mu(n)$ is defined as follow
\[ \mu(n) = \begin{cases}
 0  \text{ if $n$ has one or more repeated prime factors} \\
       1  \text{ if  $n=1$}  \\
        (-1)^k  \text{ if $n$ is a product of $k$ distinct primes}.
\end{cases} .\]  
\end{definition}
We also recall Ramanujan's sum (see \cite[Section 5.6]{hardy1979introduction}). Let $m,n$ be two positive integers. The Ramanujan sum $c_{n}(m)$ is defined as follow
\[ c_n(m)=\sum_{\substack{1 \leq a \leq n \\ \gcd(a,n)=1}} {\rm exp}(\frac{2\pi i ma}{n}). \]

If $\gcd(n,m)=a$ and $n=aN$ then by \cite[Theorem 272]{hardy1979introduction}, we have
\[ c_n(m) = \frac{\mu(N) \varphi(n)}{\varphi(N)} .\]
In particular, when $m=1$ and $n=d$, we have (see \cite[Equation 16.4.4]{hardy1979introduction})

\begin{lem} 
Let $d$ be a positive integer. Let $\omega$ be a primitive $d$th root of unity. Then 
\[ \sum_{\substack{1 \leq i \leq d\\ \gcd(i,d)=1}} \omega^i = \mu(d) .\] 

In other words, the sum of all primitive $d$th roots of unity is equal to $\mu(d).$
\end{lem}
We recall the theory of Gauss sum. Let $\zeta_{D}=\exp \left(\frac{2 \pi i}{D} \right)$ be a primitive $D$th root of unity. 
\begin{definition} (\cite[Chapter V]{ayoub1963introduction}) 
The Gauss sum $G(b, \chi)$ is defined as follow 
\[ G(b, \chi)=\sum_{a=1}^{D-1} \chi(a) \zeta_{D}^{ab} .\] 
\end{definition}
We recall the following fundamental property \cite[Theorem 4.12, page 312]{ayoub1963introduction}
\[
G(b,\chi)=\chi(b) G(1,\chi).
\]
Furthermore, by \cite[Theorem 4.17]{ayoub1963introduction}, we have $ G(1, \chi) = \sqrt{\Delta}$. Consequently $G(b,\chi)=\chi(b) \sqrt{\Delta}$.  In particular, if $\omega$ is not a primitive $D$th root of unity then 
\[ \sum_{a=1}^{D-1} \chi(a) \omega^{a} = 0.\]  
On the other hand, if $\omega$ is a primitive $D$th root of unity, namely $\omega = \zeta_{D}^b$ with $\gcd(b,D)=1$ then 
\[ \sum_{a=0}^{D-1} \chi(a) \omega^a  = \chi(b) \sqrt{\Delta} .\] 

In summary, we have the following lemma.
\begin{lem}
Let $\omega$ be a $D$th root of unity. 
\begin{enumerate}
    \item If $\omega$ is not a primitive $D$th root of unity then 
    \[ \sum_{a=1}^{D-1} \chi_{\Delta}(a) \omega^{a} = 0 .\] 
    \item If $\omega=\zeta_{D}^b$ with $\gcd(b,D)=1$ then 
    \[ \sum_{a=1}^{D-1} \chi_{\Delta}(a) \omega^{a} = \chi(b) \sqrt{\Delta} .\]

\end{enumerate}
\end{lem}
We can now simplify $\lambda(\omega).$ In fact 
\begin{align*}
    2 \lambda(\omega) &= \sum_{a=0}^{D-1} \chi(a) \omega^a + \sum_{a=0}^{D-1} \chi^2(a) \omega^a \\
    &= \sum_{a=0}^{D-1} \chi(a) \omega^a + \sum_{\substack{0 \leq a \leq D-1\\ \gcd(a,D)=1} } \omega^a.
\end{align*}
We consider two cases. \\
\textbf{Case 1.} $\omega$ is a primitive $d$th root of unity with $d<D.$ In this case 
\[ \sum_{a=0}^{D-1} \chi(a) \omega^a = 0 ,\]
and 

\[\sum_{\substack{0 \leq a \leq D-1\\ \gcd(a,D)=1} } \omega^a = \frac{\varphi(D)}{\varphi(d)} \mu(d) .\] 
Therefore 
\[ \lambda(\omega) = \frac{1}{2} \frac{\varphi(D)}{\varphi(d)} \mu(d) .\] 
\textbf{Case 2.} $\omega = \zeta_{D}^{b}$ is a primitive $D$th root of unity; namely $\omega = \zeta_{D}^b$ with $\gcd(b,D)=1$. In this case, we have 

\[\sum_{a=0}^{D-1} \chi(a) \omega^a = \chi(b) \sqrt{D}, \]
and 

\[ \sum_{\substack{0 \leq a \leq D-1\\\gcd(a,D)=1} } \omega^a = \mu(D). \]
Consequently 
\[ \lambda(\omega) = \frac{1}{2} \left[\chi(b) \sqrt{\Delta} + \mu(D) \right] .\] 
Let us write $[a]_b$ for the multiset $\{ \underbrace{a, \ldots, a}_{\text{$b$ times}} \}$. By the above calculations, we have the following conclusion. 
\begin{thm}
\label{thm:spectrum_Paleygraph}
The spectrum of the generalized Paley graph $P_{\Delta}$ is the union of the following multisets 
\[ \left[\frac{1}{2} \frac{\varphi(D)}{\varphi(d)} \mu(d) \right]_{\varphi(d)} \quad \text{for } d|D \quad \text{and } d<D, \]
and 
\[  \left[  \frac{1}{2} (\sqrt{\Delta} + \mu(D)) \right]_{\frac{\varphi(D)}{2}}, \] 
and 
\[ \left[  \frac{1}{2}(-\sqrt{\Delta} + \mu(D)) \right]_{\frac{\varphi(D)}{2}}. \]

\end{thm}
We discuss a corollary of this theorem. First, we recall that an undirected graph $G=(V,E)$ is called a bipartite graph if  $V(G)$ can be decomposed into two disjoint sets such that no two vertices within the same set are adjacent. As explained in \cite{murty2003}, if $G$ is regular of degree $r$ then $G$ is a bipartite graph if and only if $-r$ is an eigenvalue of $G.$

\begin{cor}
Suppose that $\Delta>0$. Then, the generalized Paley graph $P_{\Delta}$ is a bipartite graph if  and only if $\Delta$ is even. 
\end{cor}
\begin{proof}
If $\Delta$ is even then we can decompose $V(P_{\Delta})$ into the following two sets 
\[ V_{\text{odd}} = \{a \in V(G)|  a \equiv 1 \pmod{2} \} , \] 
and 
\[ V_{\text{even}} = \{a \in V(G)|  a \equiv 0 \pmod{2} \} .\] 
If $u,v$ both belong to either $V_{\text{odd}}$ or $V_{\text{even}}$ then $u-v$ is even. Therefore, $\chi(u-v)=0.$ By definition, $u$ and $v$ are not adjacent. We conclude that $P_{\Delta}$ is a bipartite graph. 

Conversely, let us assume that $P_{\Delta}$ is a bipartite graph. By the above remark and the fact that $P_{\Delta}$ is regular of degree $\frac{1}{2} \varphi(D)$, we conclude that one of its eigenvalues must be $-\frac{1}{2} \varphi(D).$ Clearly this eigenvalue cannot be either $\frac{1}{2} (\sqrt{\Delta} + \mu(D))$ or $\frac{1}{2} (-\sqrt{\Delta} + \mu(D))$ because these values are not integers. Consequently, by Theorem \ref{thm:Paley_Ramanujan}, there must exists $d|D$ and $d<D$ such that 

\[ \frac{1}{2} \frac{\varphi(D)}{\varphi(d)} \mu(d) = -\frac{1}{2} \varphi(D) .\] 
Equivalently, we must have $\varphi(d)=-\mu(d) \in \{-1, 1 \}.$ We conclude that $d=2$ and hence $D$ is even. 
\end{proof}

\section{Ramanujan Paley graphs}
In this section, we investigate the following question: which $P_{\Delta}$ is a Ramanujan graph? For this question to make sense, we need to assume that $P_{\Delta}$ is an undirected graph. This requires $\chi_{\Delta}$ is an even character; or equivalently $\Delta>0.$ We first recall the definition of a Ramanujan graph (see \cite{lubotzky1988ramanujan, murty2003}). 

\begin{definition}
\label{def:Ramanujan}
Let $G$ be a connected $r$-regular graph with $N$ vertices, and 
let $r=\displaystyle \lambda _{1}\geq \lambda_{2}\geq \cdots \geq \lambda_{N}$ be the eigenvalues of the adjacency matrix of $G$.  Since $G$ is connected and $r$-regular, its eigenvalues satisfy 
\( |\lambda_i| \leq r, 1 \leq i \leq N.\)
Let 
\[\lambda (G)=\max_{|\lambda_i|<r}|\lambda_{i}| .\]
The graph $G$ is a \textit{Ramanujan graph} if  
\[ \lambda (G)\leq 2{\sqrt {r -1}} .\] 
\end{definition}
It is well-known that $P_\Delta$ is a Ramanujan graph if $\Delta=p$, where $p$ is a prime number, which is congruent to 1 modulo 4. In fact, by Proposition~\ref{prop:degree}, $P_\Delta$ is a regular graph of degree $r:=\dfrac{p-1}{2}$. By Theorem~\ref{thm:spectrum_Paleygraph}, the eigenvalues of the adjacency matrix of $P_\Delta$ which are smaller than $r$ are $\dfrac{\pm\sqrt{p}-1}{2}$. Because $p\equiv 1\pmod 4$ we see that $p\geq 5$ and  $\lambda(P_\Delta)=\dfrac{\sqrt{p}+1}{2}< 2\sqrt{r-1}=2\sqrt{\dfrac{p-3}{2}}$. When $D$ is not a prime number, we have the following.

\begin{lem}
\label{lem:Paley_Ramanujan}
 Suppose $D$ is not a prime number. The graph $P_\Delta$ is a Ramanujan graph if and only if 
 \[
 \dfrac{\varphi(D)}{(p-1)^2}+\dfrac{16}{\varphi(D)}\leq 8,\]
{where $p$ is the smallest odd prime divisor of $D$},
 and
 \[
 \begin{cases}
 \dfrac{\varphi(D)}{D}\geq \dfrac{1}{8}+\dfrac{1}{4\sqrt{D}}+\dfrac{17}{8D} & \text{ if $D$ is odd}\\
 \dfrac{\varphi(D)}{D}\geq \dfrac{1}{8}+\dfrac{2}{D} & \text{ if $D$ is even}
 \end{cases}.
 \]
\end{lem}
  \begin{proof}
Let $d$ be an arbitrary divisor of $D$ such that $1<d<D$ and $d\not=2$ if $4\mid D$. One has
  \[
  \dfrac{\varphi(D)}{2\varphi(d)}\leq 2\sqrt{\dfrac{\varphi(D)}{2}-1}
  \Leftrightarrow  \dfrac{\varphi(D)^2}{\varphi(d)^2}\leq 8\varphi(D)-16 \Leftrightarrow  \dfrac{\varphi(D)}{\varphi(d)^2}+ \dfrac{16}{\varphi(D)}\leq 8.
  \]
  Clearly if $d$ has an odd prime divisor then $\dfrac{\varphi(D)}{\varphi(d)^2}+ \dfrac{16}{\varphi(D)}\leq \dfrac{\varphi(D)}{(p-1)^2}+ \dfrac{16}{\varphi(D)}$, where $p$ is the smallest odd prime divisor of $D$.
  
  One also has
  \[
  \dfrac{\sqrt{D}+1}{2}\leq 2\sqrt{\dfrac{\varphi(D)}{2}-1}
  \Leftrightarrow (1+\sqrt{D})^2\leq 8\varphi(D)-16
  \Leftrightarrow \dfrac{\varphi(D)}{D}\geq \dfrac{1}{8}+\dfrac{1}{4\sqrt{D}}+\dfrac{17}{8D},
  \]
  and
  \[
  \dfrac{\sqrt{D}}{2}\leq 2\sqrt{\dfrac{\varphi(D)}{2}-1}
  \Leftrightarrow D\leq 8\varphi(D)-16\Leftrightarrow 
  \dfrac{\varphi(D)}{D}\geq \dfrac{1}{8}+\dfrac{2}{D}.
  \]
  The lemma now follows from Theorem~\ref{thm:spectrum_Paleygraph}.
  \end{proof}
\begin{thm}
\label{thm:Paley_Ramanujan}
The graph $P_\Delta$ is a Ramanujan graph if and only if
\begin{enumerate}
    \item either     $D=8$, 
    \item or $D=4p$, where $p$ is a prime number, $p\equiv 3\pmod 4$,
    \item or $D=8p$, where $p$ is an odd prime number,
    \item  or $D=4p_1p_2$ where $p_1$ and $p_2$ are primes, $p_1p_2\equiv 3\pmod 4$ and $p_1<p_2\leq 4p_1-5$.
    \item or $D=8p_1p_2$ where $p_1$ and $p_2$ are  primes and $2<p_1<p_2\leq 2p_1-3$, 
          \item or  $D$ is a prime number $p$ with $p\equiv 1\pmod 4$,
      
      \item or $D=p_1p_2$ where $p_1$ and $p_2$ are  primes, $p_1p_2\equiv 1\pmod 4$ and $p_1<p_2\leq 8p_1-9$.
\end{enumerate}
\end{thm}
\begin{proof}
{\bf a) We first consider the case that $D$ is even.} It is easy to check that $P_8$ is a Ramanujan graph. So we suppose that  $D\not=8$ and $P_\Delta$ is a Ramanujan graph. Let $D=2^{a}p_1^{a_1}\cdots p_k^{a_k}$  be the prime factorization of $D$ with $2<p_1<\cdots<p_k$. One has    $a\in\{2,3\}$, $k\geq 1$, and $a_2=\cdots=a_k=1$.
By Lemma \ref{lem:Paley_Ramanujan},
\[
8\geq  \dfrac{\varphi(D)}{(p_1-1)^2}+\dfrac{16}{\varphi(D)}>\dfrac{\varphi(D)}{(p_1-1)^2}=2^{a-1}\dfrac{p_2-1}{p_1-1}\cdots (p_k-1).
\]
This inequality forces $k\leq 2$. In fact, if $k\geq 3$ then
\[
2^{a-1}\dfrac{p_2-1}{p_1-1}\cdots (p_k-1)>2(p_3-1)\geq 2(7-1)=12>8.
\]

\noindent\textbf{Case 1: $k=2$.}
In this case $D=2^{a}p_1p_2$, $a\in\{2,3\}$. 
One has
\[
 \dfrac{\varphi(D)}{D}\geq \dfrac{1}{8}+\dfrac{2}{D}
 \Leftrightarrow  \dfrac{3}{4}\geq \dfrac{1}{p_1}+\dfrac{1}{p_2}+\dfrac{2^{2-a}-1}{p_1p_2}.
\]
The last inequality is true because
\[
\dfrac{1}{p_1}+\dfrac{1}{p_2}+\dfrac{2^{2-a}-1}{p_1p_2} \leq \dfrac{1}{3}+\dfrac{1}{5}<\dfrac{3}{4}.
\]
\noindent\textbf{Subcase 1.1: $a=2$ ($k=2$).} In this case $D=4p_1p_2$, where $p_1<p_2$ and $p_1p_2\equiv 3\pmod 4$. By Lemma~\ref{lem:Paley_Ramanujan}, the graph $P_\Delta$ is a Ramanujan graph if and only if 
\[
\begin{aligned}
\dfrac{\varphi(D)}{(p_1-1)^2}+\dfrac{16}{\varphi(D)}\leq 8 &\Leftrightarrow   \dfrac{p_2-1}{p_1-1}+\dfrac{4}{(p_1-1)(p_2-1)}\leq 4\\
&\Leftrightarrow 4(p_1-1)-(p_2-1)\geq \dfrac{4}{p_2-1}.
\end{aligned}
\]
In the case that $p_2\geq 5$, the latter condition is equivalent to 
\[
4(p_1-1)-(p_2-1)\geq 1,
\]
which is equivalent to
\[
p_2\leq 4p_1-5.
\]

\noindent\textbf{Subcase 1.2: $a=3$ ($k=2$).} In this case $D=8p_1p_2$, where $2<p_1<p_2$. By Lemma~\ref{lem:Paley_Ramanujan}, the graph $P_\Delta$ is a Ramanujan graph if and only if 
\[
\begin{aligned}
\dfrac{\varphi(D)}{(p_1-1)^2}+\dfrac{16}{\varphi(D)}\leq 8 &\Leftrightarrow   \dfrac{p_2-1}{p_1-1}+\dfrac{1}{(p_1-1)(p_2-1)}\leq 2\\
&\Leftrightarrow 2(p_1-1)-(p_2-1)\geq \dfrac{1}{p_2-1}.
\end{aligned}
\]
 The latter condition is equivalent to 
 \[
 2(p_1-1)-(p_2-1)\geq 1,
 \]
 which is equivalent to
\[
p_2\leq 2p_1-3.
\]

\noindent\textbf{Case 2: $k=1$.}
In this case $D=2^{a}p_1$, $a\in\{2,3\}$. 
One has
\[
 \dfrac{\varphi(D)}{D}\geq \dfrac{1}{8}+\dfrac{2}{D}
 \Leftrightarrow  \dfrac{3}{4}\geq \dfrac{1+2^{2-a}}{p_1}.
\]
The last inequality is true because
\[
\dfrac{1+2^{2-a}}{p_1} \leq \dfrac{2}{3}<\dfrac{3}{4}.
\]
Note also that
\[
\dfrac{\varphi(D)}{(p_1-1)^2}+\dfrac{16}{\varphi(D)}=\dfrac{2^{a-1}+2^{5-a}}{p_1-1}\leq \dfrac{10}{3-2}=5<8.
\]

{\bf b) Now we consider the case $D$ is odd.} 
It is well known that if $\Delta$ is a prime number then $P_\Delta$ is a Ramanujan graph.(For example, see the discussion after Definition~\ref{def:Ramanujan}.)

Suppose that $D$ is not a prime number and that $P_\Delta$ is a Ramanujan graph. Let  $D=p_1^{a_1}p_2^{a_2}\cdots p_k^{a_k}$ be the prime factorization of $D$ with $2<p_1<p_2<\cdots<p_k$. Since $d$ is square-free, $a_1=\cdots=a_k=1$, and $k\geq 2$. By Lemma \ref{lem:Paley_Ramanujan},
\[
8\geq  \dfrac{\varphi(D)}{(p_1-1)^2}+\dfrac{16}{\varphi(D)}>\dfrac{p_2-1}{p_1-1}\cdots (p_k-1).
\]
This inequality forces $k\leq 3$. In fact, if $k\geq 4$ then
\[
\dfrac{p_2-1}{p_1-1}\cdots (p_k-1)>(p_3-1)(p_4-1)\geq (5-1)(7-1)=24>8.
\]

\noindent\textbf{Case 1: $k=3$.} Since $8>\dfrac{p_2-1}{p_1-1}(p_3-1)$, we conclude that $p_3=7$.  Hence  $(p_1,p_2)=(3,5)$. However, in this case, one has 
$\dfrac{p_2-1}{p_1-1}(p_3-1)>8$, a contradiction.

\noindent\textbf{Case 2: $k=2$.}
In this case $D=p_1p_2$, where $p_1<p_2$ and $p_1p_2\equiv 1\pmod 4$. One has
\[
 \dfrac{\varphi(D)}{D}\geq \dfrac{1}{8}+\dfrac{1}{4\sqrt{D}}+\dfrac{17}{8D}
 \Leftrightarrow  \dfrac{7}{8}\geq \dfrac{1}{p_1}+\dfrac{1}{p_2}+\dfrac{1}{4\sqrt{p_1p_2}}+\dfrac{4}{p_1p_2}.
\]
The last inequality is true because
\[
\dfrac{1}{p_1}+\dfrac{1}{p_2}+\dfrac{1}{4\sqrt{p_1p_2}}+\dfrac{4}{p_1p_2}\leq \dfrac{1}{3}+\dfrac{1}{5}+\dfrac{1}{4\sqrt{15}}+\dfrac{4}{15}<\dfrac{7}{8}.
\]

By Lemma~\ref{lem:Paley_Ramanujan}, the graph $P_\Delta$ is a Ramanujan graph if and only if 
\[
\begin{aligned}
\dfrac{\varphi(D)}{(p_1-1)^2}+\dfrac{16}{\varphi(D)}\leq 8 &\Leftrightarrow   \dfrac{p_2-1}{p_1-1}+\dfrac{16}{(p_1-1)(p_2-1)}\leq 8\\
&\Leftrightarrow 8(p_1-1)-(p_2-1)\geq \dfrac{16}{p_2-1}.
\end{aligned}
\]
In the case that $p_2\geq 17$, the latter condition is clearly equivalent to 
\[
8(p_1-1)-(p_2-1)\geq 1,
\]
which is equivalent to
\[
p_2\leq 8p_1-9.
\]
All pairs of primes $(p_1,p_2)$ satisfying that $p_1<p_2<17$ and $p_1p_2\equiv 1\pmod 4$ are $(3,7)$, $(3,11)$, $(7,11)$ and $(5,13)$. One can check directly that for such a pair $(p_1,p_2)$ we always have
\[
8(p_1-1)-(p_2-1)\geq \dfrac{16}{p_2-1} \text{ and } p_2\leq 8p_1-9. 
\qedhere
\]
\end{proof}
\begin{rem}
There are infinitely pairs of primes $(p_1,p_2)$ satisfying the  conditions in part (4)  of  Theorem \ref{thm:Paley_Ramanujan}. In fact, let us first recall the prime number theorem for arithmetic progressions. Let $a$, $d$ be two positive integers with $\gcd(a,d)=1$. Let 
\[
\pi(x;a,d)=|\{p \text{ prime }\mid p\leq x, p\equiv a\pmod d \}|.
\]
Then $\pi(x;a,d)\sim \dfrac{1}{\varphi(d)}\dfrac{x}{\ln x}$ as $x\to \infty$. (See for example \cite[page 257]{leveque2002topics}.) Hence for any fixed number $k>1$,
\[
\pi(kx;a,d)-\pi(x;a,d)=|\{p \text{ prime }\mid x<p\leq kx,\; p\equiv a\pmod d \}|\sim \dfrac{k-1}{\varphi(d)}\dfrac{x}{\ln x} \text{ as } x\to \infty.
\]
In particular for $x$ big enough, there exists a prime $p$ such that 
\[
x<p\leq kx,\quad p\equiv a\pmod d.
\]
Now let $p_1>5$ be a prime number congruent to 1 modulo 4  which is big enough so that there exists a prime $p_2$ such that
\[
p_1<p_2\leq 3p_1,\quad p\equiv 3\pmod 4.
\]
Then $p_1p_2\equiv 3\pmod 4$ and $p_2<4p_1-5$. Hence the pair $(p_1,p_2)$ satisfies the conditions in part (4) of Theorem~\ref{thm:Paley_Ramanujan}.

Similarly, there are infinitely pairs of primes $(p_1,p_2)$ satisfying the  conditions in part (5), or in part (7),  of  Theorem \ref{thm:Paley_Ramanujan}.
\end{rem}

\section{Cheeger number of generalized Paley graphs}
Let $\chi=\chi_{\Delta}$ be an even quadratic character (in other words, $\Delta>0$) and $P_{\Delta}$ the corresponding generalized Paley graph. In this section, we provide an upper bound for the Cheeger number of $P_{\Delta}$ (also known by other names such as isoperimetric number or edge expansion ratio). Our estimation gives a natural generalization of the results in \cite{cramer2016isoperimetric}. We also answer the authors' suspicion about the relationship between the upper bounds that they used in this paper, namely the $\alpha$-bound and the $\dfrac{p-1}{4}$-bound (see \cite[Section 2.2]{cramer2016isoperimetric}.) While our result is more general and explicit, we remark that our approach is inspired by the method in \cite{cramer2016isoperimetric}.

First, let us recall the definition of the Cheeger number of a graph. Let $G=(E,V)$ be an undirected graph. Let $F$ be a subset of $V$.  For a subset  $F \subseteq V$, the boundary of $F$, denoted by $\partial{F}$, is the set of all edges going from a vertex in $F$ to a vertex outside of $F$. The Cheeger number of $G$ is defined as 
\[ h(G) :=\min \left\{ \frac{|\partial{F}|}{|F|} \middle\vert\ F \subseteq V(G),0<|F|\leq {\tfrac {1}{2}}|V(G)| \right\}. \]

Cheeger number is of great importance in various scientific fields. For example, it has found applications in graph clustering, analysis of Markov chain, and image segmentation (see \cite{jerrum2004polynomial, kannan2004clusterings, sinclair1989approximate, spielman1996spectral}.)

Let $\{\alpha_1, \alpha_2, \ldots, \alpha_k \}$ be the set of all elements on the interval $[1, \ldots, \lfloor \frac{D}{2}]]$ such that $\chi(\alpha_i)=1.$ Note that since $\chi$ is even, $\chi(D-a)=\chi(a)$. Therefore, we can conclude that
\[ \{\alpha_1, D- \alpha_1, \alpha_2, D - \alpha_2, \ldots, \alpha_k, D -\alpha_k \} , \]
is the set of $a \in [0,D]$ such that $\chi(a)=1.$ In particular, we see that $k=\dfrac{\varphi(D)}{4}.$

The following proposition is a direct analog of \cite[Proposition 6]{cramer2016isoperimetric}.
\begin{prop}
Let $F=\{0, 1, \ldots, \lfloor \frac{D}{2} \rfloor -1 \} \subset V(P_{\Delta})$. Then $|F|=\lfloor \frac{D}{2} \rfloor$ and 
\[ |\partial{F}| = 2 \sum_{i=1}^k \alpha_i .\]  
\end{prop}
\begin{proof}
By definition, an element of $\partial{F}$ will have the form $(i, i+a)$ where $i \in F$, $i+a \not \in F$, and $\chi(a)=1.$ In particular, we know that 
\[ a \in \{\alpha_1, D- \alpha_1, \alpha_2, D - \alpha_2, \ldots, \alpha_k, D -\alpha_k \} .\] 
Let us fix an index $i$ where $1 \leq i \leq \dfrac{\varphi(D)}{4}.$ The number of edges of the form $(j, j+\alpha_i)$ is equal to the number of $j \in F$ such that 
\begin{equation} \label{eq:first_inequality}
    (j+\alpha_i) \pmod{D} \geq \lfloor \frac{D}{2} \rfloor.
\end{equation} 
Because $\alpha \leq \lfloor \frac{D}{2} \rfloor$, we conclude that the number of $j \in F$ satisfying the inequality \ref{eq:first_inequality} is exactly $\alpha_i.$ Similarly, the number of edges of the form $(j, j+D-\alpha_i)$ is equal to the number of $j \in F$ such that 
\begin{equation} \label{eq:second_inequality}
    (j+D-\alpha_i) \pmod{D} =(j-\alpha_i) \pmod{D} \geq \lfloor \frac{D}{2} \rfloor 
\end{equation} 
Again, the number of $j$ satisfying the inequality \ref{eq:second_inequality} is exactly $\alpha_i$. Summing up over all $i$, we have 

\[ |\partial{F}| = 2 \sum_{i=1}^k \alpha_i .
\qedhere\]  
\end{proof}

By definition, we conclude that 
\[ h(P_{\Delta}) \leq \dfrac{2}{\lfloor D/2\rfloor}  \sum_{i=1}^k \alpha_i .\] 
We can further simplify the right-hand side of the above estimate using zeta values. In order to do so, we first recall the definition of the $L$-function $L(\chi, s)$ attached to $\chi$ 
\[ L(\chi,s)= \sum_{n=1}^{\infty} \frac{\chi(n)}{n^s} .\] 
This $L$-function is convergent for $s \in \C$ such that $\Re(s) >0$ and it is absolutely convergent if $\Re(s)>1.$ Furthermore, there is an Euler product formula
\[ L(\chi, s) = \prod_{p} \dfrac{1}{1-\chi(p)p^{-s}} ,\] 
where $p$ runs over the set of all prime numbers. This Euler product formula shows that  $L(\chi,s)>0$ if $s \in \R$ and $s>1.$ We are now ready to state the following theorem.

\begin{thm} 
\[
h(P_\Delta)\leq \alpha:=
\dfrac{1}{8\lfloor D/2\rfloor}\left( D\varphi(D)-\mu(D) \varphi(D)-\dfrac{8D\sqrt{D}}{\pi^2}\left(1-\dfrac{\chi(2)}{4}\right)L(2,\chi)\right).
\]

\end{thm}
\begin{proof} One has
\[
2\sum_{i=1}^k \alpha_i =\sum_{a=1, \gcd(a,D)=1}^{\lfloor D/2\rfloor} (1+\chi(a))a=\sum_{a=1, \gcd(a,D)=1}^{\lfloor D/2\rfloor} a +  \sum_{a=1}^{\lfloor D/2\rfloor}\chi(a)a.
\]

By \cite[Theorem 13.1]{berndt1976classical}, 
\[
\sum_{a=1}^{\lfloor D/2\rfloor}\chi(a)a =-\dfrac{D\sqrt{D}}{\pi^2}\left(1-\dfrac{\chi(2)}{4}\right)L(2,\chi).
\]

By \cite[Theorem]{Baum1982}, 
\[
\sum_{a=1, \gcd(a,D)=1}^{\lfloor D/2\rfloor} a=\dfrac{1}{8}(D\varphi(D)-\epsilon\psi(D)),
\]
where $\epsilon=1$ if $D$ is odd and $\epsilon=0$ if $D$ is even; and $\psi(D)=\prod\limits_{p\mid D}(1-p)$. Note that in the case that $D$ is odd, $D$ is then square-free, hence $\psi(D)=\mu(D)\varphi(D)$. Also if $D$ is even then $D$ is divisible by $4$, hence $\mu(D)=0$.
Thus
\[
\sum_{a=1, \gcd(a,D)=1}^{\lfloor D/2\rfloor} a=\dfrac{1}{8}(D\varphi(D)-\mu(D)\varphi(D)).
\]
Therefore
\[
\alpha:=\dfrac{2\sum\limits_{i=1}^k\alpha_i}{\lfloor D/2\rfloor}=\dfrac{1}{8\lfloor D/2\rfloor}\left( D\varphi(D)-\mu(D) \varphi(D)-\dfrac{8D\sqrt{D}}{\pi^2}\left(1-\dfrac{\chi(2)}{4}\right)L(2,\chi)\right).
\]
\end{proof}

\begin{rem} Let the notation be as above. If 
 $D$ is even, then
\[
\alpha= \dfrac{\varphi(D)}{4}-\dfrac{2\sqrt{D}}{\pi^2}\left(1-\dfrac{\chi(2)}{4}\right)L(2,\chi).
\]

If  $D$ is odd, then
\[
\begin{aligned}
\alpha&= \dfrac{1}{4(D-1)}\left( D\varphi(D)- \mu(D)\varphi(D)-\dfrac{8D\sqrt{D}}{\pi^2}\left(1-\dfrac{\chi(2)}{4}\right)L(2,\chi)\right)\\
&= \dfrac{\varphi(D)}{4}-\dfrac{1}{4(D-1)}\left(\dfrac{8D\sqrt{D}}{\pi^2}(1-\dfrac{\chi(2)}{4})L(2,\chi)-(1-\mu(D))\varphi(D)\right). 
\end{aligned}
\]
\end{rem}

\begin{cor} \label{cor:estimate} Let the notation be as above. Then 
$\alpha<\dfrac{\varphi(D)}{4}$.
\end{cor}
\begin{proof}
If $D$ is even then the statement follows from the previous corollary and the fact that $L(2,\chi)>0$.
So we suppose that $D$ is odd. One can check that for $D=5$ or $D=13$ then the statement is true. So we suppose further that $D\geq 17$. 

\noindent {\bf Case 1: } $\chi(2)=1$. In this case, 
\[
L(2,\chi)=\sum_{n=1}^\infty\dfrac{\chi(n)}{n^2}\geq 1+\dfrac{1}{2^2}-\sum_{n=3}^\infty\dfrac{1}{n^2}=\dfrac{5}{2}-\dfrac{\pi^2}{6}>\dfrac{5}{6}.
\]

Hence
\[
\begin{aligned}
\dfrac{8D\sqrt{D}}{\pi^2}(1-\dfrac{\chi(2)}{4})L(2,\chi)-(1-\mu(D))\varphi(D)>\dfrac{8D\sqrt{D}}{\pi^2}
\end{aligned}\cdot \dfrac{3}{4}\cdot\dfrac{5}{6}-2D>0.
\]
Thus $\alpha<\dfrac{\varphi(D)}{4}$.

\noindent {\bf Case 2: } $\chi(2)=1$. In this case, 
\[
L(2,\chi)=\sum_{n=1}^\infty\dfrac{\chi(n)}{n^2}\geq 1-\sum_{n=2}^\infty\dfrac{1}{n^2}=2-\dfrac{\pi^2}{6}>\dfrac{1}{3}.
\]

Hence
\[
\begin{aligned}
\dfrac{8D\sqrt{D}}{\pi^2}(1-\dfrac{\chi(2)}{4})L(2,\chi)-(1-\mu(D))\varphi(D)>\dfrac{8D\sqrt{D}}{\pi^2}
\end{aligned}\cdot \dfrac{5}{4}\cdot\dfrac{1}{3}-2D>0.
\]
Thus $\alpha<\dfrac{\varphi(D)}{4}$.
\end{proof}

In the particular case when $D$ is a prime of the form $4k+1$, Corollary \ref{cor:estimate} shows that the $\alpha$-bound discussed in \cite{cramer2016isoperimetric} is better than the $\dfrac{p-1}{4}$-bound.

\begin{rem}
It is interesting to see whether the $\alpha$-bound provides the optimal value for $h(P_{\Delta})$; namely $h(P_{\Delta})=\alpha.$ This is true in the case $P_{\Delta}$ is a cycle graph (by Corollary \ref{cor:cycle}, this happens if $\Delta \in \{5,8, 12 \}$). The reason is that in this case $\alpha=\dfrac{4}{D}$ if $D \in \{8,12\}$ and $\alpha=\dfrac{4}{D-1}$ if $D=5$ which is precisely the Cheeger number of $P_{\Delta}$ (by the result in \cite{krebs2011expander}, it is known  that for a cycle graph of  order $n$ its Cheeger number is $\dfrac{4}{n}$ if $n$ is even and $\dfrac{4}{n-1}$ if $n$ is odd.) 
\end{rem}
\section*{Data Availability Statement} Data sharing not applicable to this article as no datasets were generated or analysed during the current study.
\section*{Acknowledgements} The third named author is very grateful to Professor Moshe Rosenfeld who kindled his interest in using number theory to attack problems in graph theory and combinatorics.  

\bibliographystyle{abbrv}
\bibliography{references.bib}

\end{document}